\documentclass[12pt, reqno]{amsart}
\usepackage{ amsmath,amsthm, amscd, amsfonts, amssymb, graphicx, color}
\usepackage[bookmarksnumbered, colorlinks, plainpages]{hyperref}
\newtheorem{thm}{Theorem}[section]
\newtheorem{cor}[thm]{Corollary}
\newtheorem{lem}[thm]{Lemma}

\newtheorem{exam}[thm]{Example}
\numberwithin{equation}{section}

\begin{document}

\title{Generalized Zhou inverses in rings}

\author{Huanyin Chen}
\author{Marjan Sheibani Abdolyousefi}

\address{
Department of Mathematics\\ Hangzhou Normal University\\ Hang -zhou, China}
\email{<huanyinchen@aliyun.com>}
\address{
Women's University of Semnan (Farzanegan), Semnan, Iran}
\email{<sheibani@fgusem.ac.ir>}


\begin{abstract}
We introduce and study a new class of generalized inverses in rings.
An element $a$ in a ring $R$ has generalized Zhou inverse if there exists $b\in R$ such that $bab=b, b\in comm^2(a), a^n-ab\in \sqrt{J(R)}$ for some $n\in \mathbb{N}$. We prove that $a\in R$ has generalized Zhou inverse if and only if there exists $p=p^2\in comm^2(a)$ such that $a^n-p\in \sqrt{J(R)}$ for some $n\in \mathbb{N}$. Cline's formula and Jacobson's Lemma for generalized Zhou inverses are established. In particular, the Zhou inverse in a ring is characterized.\end{abstract}

\maketitle

\section{Introduction}

Let $R$ be an associative ring with an identity. The commutant of $a\in R$ is defined by $comm(a)=\{x\in
R~|~xa=ax\}$. The double commutant of $a\in R$ is defined
by $comm^2(a)=\{x\in R~|~xy=yx~\mbox{for all}~y\in comm(a)\}$. An element $a$ in $R$ is said to have p-Drazin inverse if there exists $b\in R$ such that $$b=bab, b\in comm^2(a), a-a^2b\in \sqrt{J(R)}.$$ If the preceding $b$ exists, it is unique and is called the p-Drazin inverse of $a$.
We denote it by $a^{\ddag}$. Here, $\sqrt{J(R)}=\{x\in R~|~x^n\in J(R)~\mbox{for some}~n\in \mathbb{N}\}$. The generalized Drazin inverse is defined by replacing
$\sqrt{J(R)}$ by the set of all nilpoitent elements in $R$. The $p$-Drazin and generalized Drazin inverses were extensively studied in matrix theory and Banach algebra (see {CM}, {CC}, {LCC},{W}, {Z2}, {ZC}, {ZC1}).

The motivation of this paper is to introduce a new kind of generalized inverses characterized by its Jacobson radical. We call an element $a\in R$ has generalized Zhou inverse if there exists $b\in R$ such that $$b=bab, b\in comm^2(a), a^n-ab\in \sqrt{J(R)}.$$ We shall prove that preceding $b$ is unique, if such element exists. It will be denoted by $a^{z}$, and called the generalized Zhou inverse of $a$. We see that such generalized inverses not only behave like p-Drazin inverses, but also characterize those elements in the ring firstly studied by Zhou et al. (see {K}, {Z}).

In Section 2, we prove that an element $a\in R$ has generalized Zhou inverse if and only if there exists $p=p^2\in comm^2(a)$ such that $a^n-p\in \sqrt{J(R)}$. Let $A$ be a Banach algebra and $a\in A$, we prove that the preceding double commuant can be replaced by commutant of $a$.
Let $a,b\in R$. Then $ab$ has p-Drazin inverse if and only if $ba$ has p-Drazin inverse (see [Theorem 3.6]{W}. This was known as Cline's formula for p-Drazin inverses. Jacobson's Lemma states that $1-ab$ has generalized Drazin inverse if and only if $1-ba$ has generalized Drazin inverse.
Cline's formula and Jacobson's Lemma play important roles in the generalized inverse of a matrix over rings. In Section 3, Cline's formula and Jacobson's Lemma for the generalized Zhou inverse are established. In {K}, Zhou et al. investigated the ring with $x-x^{n+1}$ is nilpotent for any $x\in R$. Finally, in the last section, we introduce Zhou inverse and characterize an element in a ring with $a-a^{n+1}$ is nilpotent.

Throughout the paper, all rings are associative with an identity and all Banach algebras are complex. We use $N(R)$ to denote the set of all nilpotent elements in $R$.
$U(R)=\{ x\in R | $ there exists a $b\in R$ with $xb=bx=1\}.$ $J(R)$ is the Jacobson radical of a ring $R$. $\mathbb{N}$ stands for the set of all natural numbers.

\section{Representations of Generalized Zhou inverses}

The goal of this section is to explore the generalized Zhou inverse in a ring. The explicit results in a Banach algebra are obtained as well.
We begin with

\begin{lem} Let $R$ be a ring and $a\in \sqrt{J(R)}$. If $e^2=e\in comm(a)$, then $ae\in \sqrt{J(R)}$.\end{lem}
\begin{proof} Straightforward.\end{proof}

\begin{thm} Let $R$ be a ring and $a\in R$. If $a$ has generalized Zhou inverse, then it has p-Drazin inverse.\end{thm}
\begin{proof} Let $x=a^z$. Then $a^n-ax\in \sqrt{J(R)}, x\in comm^2(a)$ and $xax=x$. As $a\in comm(a)$, we see that
$ax=xa$. Since $ax\in R$ is an idempotent, $a^n-a^nx^n=a^n-ax\in \sqrt{J(R)}$, and so $$a^n(1-ax)=(a^n-a^nx^n)(1-ax)\in \sqrt{J(R)},$$ by Lemma 2.1.
Further, $a^{n-1}(a-a^2x)=a^n(1-ax)\in \sqrt{J(R)}$. By using Lemma 2.1 again, $(a-a^2x)^n=a^{n-1}(a-a^2x)(1-ax)\in \sqrt{J(R)}$.
This implies that $a-a^2x\in \sqrt{J(R)}$. Therefore $a$ has p-Drazin inverse $x$, as desired.\end{proof}

\begin{cor} Let $R$ be a ring and $a\in R$. Then $a$ has at most one generalized Zhou inverse inverse in $R$, and if the generalized Zhou inverse
of $a$ exists, it is exactly the p-Drazin inverse of $a$, i.e., $a^{z}=a^{\ddag}$.\end{cor}
\begin{proof} In view of [Theorem 3.2]{W}, every element in $R$ has at most one p-Drazin inverse, we obtain the result by Theorem 2.2.\end{proof}

\begin{exam} $(1)$ Every element in $\mathbb{Z}_5$ has generalized Zhou inverse.

$(2)$ Every element in $\mathbb{C}$ has p-Drazin inverse, but $2\in \mathbb{C}$ has no generalized Zhou inverse.\end{exam}
\begin{proof} $(1)$ Clearly, every nonzero element in $\mathbb{Z}_5$ has order $4$. Thus, we easily check that every $a\in \mathbb{Z}_5$ has
generalized Zhou inverse $a^3$.

$(2)$ Since $\mathbb{C}$ is a field with zero Jacobson radical, every complex number in $\mathbb{C}$ has p-Drazin inverse. If $2$ has a generalized Zhou inverse $x$ in $\mathbb{C}$, then $2^n=2x, x=2x^2$ for some $n\in \mathbb{N}$, and so $x=0, \frac{1}{2}$. This implies that $2^n=0$ or $1$, a contradiction. Therefore $2\in \mathbb{C}$ has no generalized Zhou inverse.\end{proof}

Generalized Zhou inverses in a ring form a proper subclass of p-Drazin inverses which were characterized by its polar-like property.

\begin{lem} Let $\mathcal{A}$ be a Banach algebra, $a,b\in A$ and $ab=ba$.\end{lem}
\begin{enumerate}
\item [(1)]{\it If $a,b\in \sqrt{J(\mathcal{A})}$, then $a+b\in \sqrt{J(\mathcal{A})}$.}
\vspace{-.5mm}
\item [(2)]{\it If $a$ or $b\in \sqrt{J(\mathcal{A})}$, then $ab\in \sqrt{J(\mathcal{A})}$.}
\end{enumerate}
\begin{proof} See [Lemma 2.4]{ZC1}.\end{proof}

We are ready to prove:

\begin{thm} Let $R$ be a ring and $a\in R$. Then the following are equivalent:\end{thm}
\begin{enumerate}
\item [(1)]{\it $a$ has generalized Zhou inverse.}
\item [(2)]{\it There exists $p^2=p\in comm^2(a)$ such that $a^n-p\in \sqrt{J(R)}$.}
\item [(3)]{\it There exists $x\in comm^2(a)$ such that $x=xax, a-a^{n+2}x\in \sqrt{J(R)}$ for some $n\in \mathbb{N}$.}
\end{enumerate}
\begin{proof} $(1)\Rightarrow (2)$ By hypothesis, there exists $b\in R$ such
that $$b=bab, b\in comm^2(a), a^n-ab\in \sqrt{J(R)}$$ for some $n\in \mathbb{N}$. Set $p=ab$. Then $p^2=p$.
Let $xa=ax$. Then $xb=bx$, and so $xp=x(ab)=(ab)x$; hence, $p\in comm^2(a)$. Furthermore, $a^n-p\in \sqrt{J(R)}$, as required.

$(2)\Rightarrow (1)$ Since $a^n-p\in \sqrt{J(R)}$, we see that $a^n+1-p\in U(R)$. Set $e=1-p$ and $b=(a^n+1-p)^{-1}(1-e)$.
Let $xa=ax$. Then $px=xp$ and $xa^{n-1}=a^{n-1}x$. Hence $x(a^{n-1}b)=(a^{n-1}b)x$, i.e., $a^{n-1}b\in comm^2(a)$.
One checks that $$\begin{array}{lll}
a^n-a(a^{n-1}b)&=&a^n-a^n(a^n+1-p)^{-1}(1-e)\\
&=&a^n-(a^n+e)(a^n+1-p)^{-1}(1-e)\\
&=&a^n-p\\
&\in& \sqrt{J(R)}.
\end{array}$$
Clearly, $a^n+e\in U(R)$, and so $$\begin{array}{lll}
(a^{n-1}b)a(a^{n-1}b)&=&a^{2n-1}(a^n+1-p)^{-2}(1-e)\\
&=&a^{n-1}(a^n+e)(a^n+1-p)^{-2}(1-e)\\
&=&a^{n-1}(a^n+1-p)^{-1}(1-e)\\
&=&a^{n-1}b.
\end{array}$$
Therefore $a$ has generalized Zhou inverse $a^{n-1}b$, as asserted.

$(1)\Rightarrow (3)$ Let $a^z=x$. In view of Theorem 2.2, we have $a-a^2x\in \sqrt{J(R)}$.
Moreover, we have $n\in \mathbb{N}$ such that $a^n-ax\in \sqrt{J(R)}, x\in comm^2(a)$ and $xax=x$. Clearly,
one checks that $$a-a^{n+2}x=(a-a^2x)-(a^n-ax)a^2x.$$ By virtue of Lemma 2.5,
$a-a^{n+2}x\in \sqrt{J(R)}$, as desired.

$(3)\Rightarrow (1)$ By hypothesis, there exists $x\in comm^2(a)$ such that $x=xax, a-a^{n+2}x\in \sqrt{J(R)}$
for some $n\in \mathbb{N}$. Clearly, $$a^n-ax=(a-a^{n+2}x)(a^{n-1}-x-a^nx).$$ By Lemma 2.5, $a^n-ax\in \sqrt{J(R)}$, the result follows.\end{proof}

\begin{cor} Let $R$ be a ring, and let $a\in R$. Then the following are equivalent:\end{cor}
\begin{enumerate}
\item [(1)]{\it $a$ has generalized Zhou inverse.}
\vspace{-.5mm}
\item [(2)]{\it There exists a unique $b\in comm^2(a)$ such
that
$b=bab, a^n-ab\in \sqrt{J(R)}.$}\vspace{-.5mm}
\item [(3)]{\it There exists a unique idempotent $p\in comm^2(a)$ such
that $a^n-p\in \sqrt{J(R)}.$}\vspace{-.5mm}
\end{enumerate}
\begin{proof} $(1)\Rightarrow (2)$ By hypothesis, there exists an $x\in R$ such
that $$x=xax, x\in comm^2(a), a^n-ax\in \sqrt{J(R)}.$$ Suppose that there exists $b\in R$ such
that
$$b=bab, b\in comm^2(a), a^n-ab\in \sqrt{J(R)}.$$
Let $p=1-ax$ and $q=1-ab$. Then
$$p^2=p, q^2=q\in comm^2(a)~\mbox{and}~a^n+p, a^n+q\in U(R).$$ By virtue of Lemma 2.1, $a^np, a^nq\in \sqrt{J(R)}.$ Hence,
$$\begin{array}{lll}
1-(1-p)q&=&1-(1-p)(a^n+p)^{-1}(a^n+p)q\\
&=&1-(1-p)(a^n+p)^{-1}a^nq\\
&=&1-ca^nq,
\end{array}$$ where $c=(1-p)(a^n+p)^{-1}$. Since $p\in comm^2(a)$, we
have $c\in comm(a^nq)$. It follows from $a^nq\in \sqrt{J(R)}$ that
$1-ca^nq\in U(R)$. Therefore
$$1-(1-p)q=1-(1-p)^2q^2=\big(1-(1-p)q\big)\big(1+(1-p)q\big),$$ and
so $1+(1-p)q=1$. This implies that $q=pq$. Likewise, $p=qp$.
Accordingly, $p=pq=qp=q$. We compute that
$$(a^n+p)^{-1}(1-p)=(a^n+p)^{-1}a^nx^n(1-p)=x^n(1-p)=x^na^nx^n=x^n.$$ Likewise, we have
$(a^n+q)^{-1}(1-q)=b^n$. Therefore $x=x^2a=(x^n)a^{n-1}=(b^n)a^{n-1}=b$, as required.

$(2)\Rightarrow (3)$ Suppose that there exists an idempotent $p\in R$ such
that
$p\in comm^2(a), a^n-p\in \sqrt{J(R)}$. Then $1-(a^n-p)\in U(R)$, and so $a-(1-p)\in U(R)$.
Set $c=(a-1+p)^{-1}p$. Then $c\in comm^2(a).$ Hence,
$$cac=(a-1+p)^{-1}p(a-1+p)(a-1+p)^{-1}=c$$ and $$a^n-ac=a^n-(a-1+p)(a-1+p)^{-1}p=a^n-p\in \sqrt{J(R)}.$$ By the uniqueness, we get
$b=c$, and so $p=(a-1+p)c=ac=ab$, as desired.

$(3)\Rightarrow (1)$ This is obvious by Theorem 2.6.\end{proof}

For a Banach algebra $\mathcal{A}$, we characterize the generalized Zhou inverse by replacing double commutants by commutants in Theorem 2.6.

\begin{lem} Let $\mathcal{A}$ be a Banach algebra. and let $a\in \mathcal{A}$. Then the following are equivalent:\end{lem}
\begin{enumerate}
\item [(1)]{\it $a$ has generalized Zhou inverse.}
\vspace{-.5mm}
\item [(2)]{\it There exists $b\in \mathcal{A}$ such
that
$$b=bab, ba=ab, a^n-ab\in \sqrt{J(\mathcal{A})}.$$}\vspace{-.5mm}
\item [(3)]{\it There exists an idempotent $p\in \mathcal{A}$ such
that $$pa=ap, a^n-p\in \sqrt{J(\mathcal{A})}.$$}
\end{enumerate}
\begin{proof} $(1)\Rightarrow (2)$ This is trivial.

$(2)\Rightarrow (3)$ This is clear by choosing $p=ab$.

$(3)\Rightarrow (1)$ Since $a^n-p\in \sqrt{J(\mathcal{A})}$ and $p\in comm(a)$, we have
$$a^n+1-p\in U(\mathcal{A}), a^n(1-p)\in \sqrt{J(\mathcal{A})}.$$ In light of [Remark 5.1]{W}, we see that
$1-p\in comm^2(a^n)$, and so $p\in comm^2(a)$. This completes the proof by Theorem 2.6.\end{proof}

\begin{thm} Let $\mathcal{A}$ be a Banach algebra, and let $a\in \mathcal{A}$. Then the following are equivalent:\end{thm}
\begin{enumerate}
\item [(1)]{\it $a$ has generalized Zhou inverse.}
\vspace{-.5mm}
\item [(2)]{\it $a-a^{n+1}\in \sqrt{J(\mathcal{A})}$ for some $n\in \mathbb{N}$.}
\end{enumerate}
\begin{proof} $(1)\Rightarrow (2)$ In view of Lemma 2.8, there exists $b\in \mathcal{A}$ such
that $$b=bab, ba=ab, w:=a^n-ab\in \sqrt{J(\mathcal{A})}$$ for some $n\in \mathbb{N}$. Hence $a^n=ab+w$, and so $a^{n-1}(a-a^{n+1})=a^n-a^{2n}=-(2ab+w)w.$
Therefore $(a-a^{n+1})^n=-(2ab+w)(1-a^n)^{n-1}w\in \sqrt{J(\mathcal{A})}$. We infer that $a-a^{n+1}\in \sqrt{J(\mathcal{A})},$ as required.

$(2)\Rightarrow (1)$ By hypothesis, we have some $n\in \mathbb{N}$ such that $a-a^{n+1}\in\sqrt{J(\mathcal{A})}$. Write $(a-a^{n+1})^m\in J(\mathcal{A})$ for some $m\in \mathbb{N}$. Then $a^m(1-a^n)^m\in J(\mathcal{A})$, and so $[a^n(1-a^n)]^m=(a-a^{n+1})^m\in J(\mathcal{A})$.
$$\begin{array}{lll}
1&=&[a^n+(1-a^n)]^{2m}\\
&=&
\left(
\begin{array}{c}
2m\\
0
\end{array}
\right)(a^n)^{2m}+\left(
\begin{array}{c}
2m\\
1
\end{array}
\right)(a^n)^{2m-1}(1-a^n)+\cdots \\
&+&\left(
\begin{array}{c}
2m\\
m
\end{array}
\right)(a^n)^{m}(1-a^n)^m+\left(
\begin{array}{c}
2m\\
m+1
\end{array}
\right)(a^n)^{m-1}(1-a^n)^{m+1}\\
&+&\left(
\begin{array}{c}
2m\\
m+2
\end{array}
\right)(a^n)^{m-2}(1-a^n)^{m+2}+\cdots +\left(
\begin{array}{c}
2m\\
2m
\end{array}
\right)(1-a^n)^{2m}.
\end{array}$$ Let $$e=\sum\limits_{i=0}^{m}C_{2m}^{i}(a^n)^{2m-i}(a-a^n)^i.$$ Then $e^2-e\in J(\mathcal{A})$. By [Lemma 2.3.8]{R}, We can find some $f^2=f\in comm^2(e)$ such that $e-f\in J(\mathcal{A})$. Hence, $f\in comm(a)$. Moreover,
$$a^{2nm}-f=(a^{2nm}-e)+(e-f),$$ and so $(a^{2nm}-f)^m\in J(\mathcal{A})$. In light of Lemma 2.8,
we complete the proof.\end{proof}

\section{Multiplicative and Additive Results}

Let $a,b,c,d\in R$ satisfy $bdb=bac, dbd=acd$. If $ac$ has p-Drazin inverse, then so has $bd$ (see [Theorem 3.1]{CM}).
We now generalize Cline's formula from p-Drazin inverses to generalized Zhou inverses.

\begin{thm} Let $R$ be a ring, and let $a,b,c,d\in R$ satisfy $bdb=bac, dbd=acd$. If $ac$ has generalized Zhou inverse, then $bd$ has generalized Zhou inverse. In this case, $(bd)^{z}=b((ac)^{z})^2d$.
\end{thm}
\begin{proof} Set $(ac)^{z}=x$.
Then we have $$x=x(ac)x, x\in comm^2(ac), (ac)^n-(ac)x\in \sqrt{J(R)}.$$
Let $e=bx^2d$. According to [Theorem 3.1]{CM},
$bd\in R^{\ddag}$ and $(bd)^{\ddag}=e$.
Therefore we have $e\in comm^2(bd), e=e(bd)e$. Let $p=(ac)^{n-1}-x$. Then $$(pa)c=(ac)^n-(ac)x$$
  that is contained in $\sqrt{J(R)}$. Furthermore, we have
  $$\begin{array}{lll}
  (bd)^n-(bd)e&=&b(db)^{n-1}d-bdbx^2d\\
 &=&b(ac)^{n-1}d-bacx^2d\\
 &=&b[(ac)^{n-1}-x]d.
  \end{array}$$ Hence,
  $$\begin{array}{lll}
  bd[(bd)^n-(bd)e]&=&bdb[(ac)^{n-1}-x]d\\
 &=&bac[(ac)^{n-1}-x]d\\\\
 &=&b[(ac)^{n}-(ac)x]d.
  \end{array}$$
    By hypothesis, we have $(db)(ac)=dbac=dbdb=(ac)(db).$ As $x\in comm^2(ac)$, we get $(db)x=x(db)$.
  Then $$db[(ac)^{n}-(ac)x]\in \sqrt{J(R)}.$$
  We note that $xy\in \sqrt{J(\mathcal{A})}$ if and only if $yx\in \sqrt{J(R)}.$
Hence, $$b[(ac)^{n}-(ac)x]d\in \sqrt{J(R)}.$$ This implies that $bd[(bd)^n-(bd)e]\in \sqrt{J(R)}.$
Therefore $$[(bd)^n-(bd)e]^2=bd[(bd)^n-(bd)e][(bd)^{n-1}-e]\in \sqrt{J(R)}.$$
Accordingly, $e$ is the generalized Zhou inverse of $bd$ and as the generalized Zhou inverse is unique we deduce that $e=bx^2d=(bd)^{z}.$\end{proof}

\begin{cor} Let $R$ be a ring, and let $a,b\in R$. If $(ab)^k$ has generalized Zhou inverse, then so does $(ba)^k$.
\end{cor}
\begin{proof} Since $(ab)^k=a(b(ab)^{k-1})$, it follows by Theorem 3.1 that $(b(ab)^{k-1})a$ has generalized Zhou inverse. Then
$(ba)^k=(b(ab)^{k-1})a$ has generalized Zhou inverse, as asserted.\end{proof}

Let $a,b,c,d\in R$ satisfy $bdb=bac, dbd=acd$. If $1-ac$ has generalized Drazin inverse, then so has $1-bd$ (see [Theorem 3.1]{M}).
We next generalize this result from generalized Drazin inverse to generalized Zhou inverses.

\begin{thm} Let $R$ be a ring, and let $a,b,c,d\in R$ satisfy $bdb=bac, dbd=acd$. If $1-ac$ has generalized Zhou inverse, then  $1-bd$ has generalized Zhou inverse. In this case, $$\begin{array}{ll}
&(1-bd)^{z}\\
=&1+b[1-(1-ac)^{z}-(1+ac)^{\pi}(1-(1-ac)^{\pi})(1-ac))^{-1}]d.
\end{array}$$
\end{thm}
\begin{proof} Since $\alpha:=1-ac$ has generalized Zhou inverse,
there exists some $n\in \mathbb{N}$ such that
$$\alpha^z\in comm^2(\alpha), \alpha^z=\alpha^z\alpha\alpha^z, \alpha^n-\alpha\alpha^z\in \sqrt{J(R)}.$$
By virtue of Corollary 2.3, $\alpha$ has p-Drazin inverse. Hence it has generalized Drazin inverse.
In light of [Theorem 2.5]{M}, $1-bd$ has generalized Drazin inverse and
$$\begin{array}{ll}
&(1-bd)^{d}\\
=&1+b[1-(1-ac)^{z}-(1+ac)^{\pi}(1-(1-ac)^{\pi})(1-ac))^{-1}]d.
\end{array}$$ This shows that $(1-bd)^d\in comm^2(1-bd)$ and $(1-bd)^d=[(1-bd)^d]^2(1-bd)$.
As in the proof of [Theorem 2.5]{M}, we check that
$$(1-bd)(1-bd)^d=1-b\alpha^{\pi}(1-\alpha^{\pi}\alpha)^{-1}d.$$
Moreover, we see that
$$\begin{array}{ll}
&(1-bd)^n-(1-bd)(1-bd)^d\\
=&(1-bd)^n-1+b\alpha^{\pi}(1-\alpha^{\pi}\alpha)^{-1}d\\
=&-bd\sum\limits_{i=0}^{n-1}(1-bd)^i+b\alpha^{\pi}(1-\alpha^{\pi}\alpha)^{-1}d\\
=&-b\big[1-\alpha^{\pi}\alpha+\sum\limits_{i=1}^{n-1}(1-ac)^i(1-\alpha^{\pi}\alpha)-\alpha^{\pi}\big](1-\alpha^{\pi}\alpha)^{-1}d\\
=&-b\big[\sum\limits_{i=0}^{n-1}(1-ac)^i(1-\alpha^{\pi}\alpha)-\alpha^{\pi}\big](1-\alpha^{\pi}\alpha)^{-1}d\\
=&-b\big[\sum\limits_{i=0}^{n-1}(1-ac)^i-\alpha^{\pi}-\sum\limits_{i=0}^{n-1}(1-ac)^i\alpha^{\pi}\alpha\big](1-\alpha^{\pi}\alpha)^{-1}d.
\end{array}$$
Then we have $$\begin{array}{ll}
&(1-ac)^n-(1-ac)(1-ac)^d\\
=&(1-ac)^n-1+\alpha^{\pi}\\
=&-\sum\limits_{i=0}^{n-1}(1-ac)^i+\alpha^{\pi}\\
\end{array}$$
Therefore $$\begin{array}{l}
(1-bd)^n-(1-bd)(1-bd)^d=\\
b\big[(1-ac)^n-(1-ac)(1-ac)^d+\sum\limits_{i=0}^{n-1}(1-ac)^i\alpha^{\pi}\alpha\big](1-\alpha^{\pi}\alpha)^{-1}d.
\end{array}$$
As in the proof of Theorem 3.1, we get $(db)(ac)=(ac)(db)$, and so $(bd)\alpha=\alpha(bd)$ and $(bd)\alpha^d=\alpha^d(bd)$.
Since $\alpha$ has p-Drazin inverse by Theorem 2.2, we have $\alpha\alpha^{\pi}\in \sqrt{J(R)}.$  Thus $$db\big[(1-ac)^n-(1-ac)(1-ac)^d+\sum\limits_{i=0}^{n-1}\alpha^{i+1}\alpha^{\pi}\big](1-\alpha^{\pi}\alpha)^{-1}\in \sqrt{J(R)}.$$
Accordingly, $$(1-bd)^n-(1-bd)(1-bd)^d\in \sqrt{J(R)}.$$ By Theorem 2.2 and [Theorem 2.5]{M}, $(1-bd)^z=(1-bd)^{\ddag}=(1-bd)^{d}$, as asserted.\end{proof}

Let $x,y\in R$. If $1-xy$ has generalized Zhou inverse then so has $1-yx$ for any $x,y\in R$. Furthermore, we derive

\begin{cor} Let $R$ be a ring, and let $a,b\in R$. If $(1-ab)^k$ has generalized Zhou inverse, then so does $(1-ba)^k$.
\end{cor}
\begin{proof}  Since $b(ab)^m=(ba)^mb$ for all $m\in \mathbb{N}$, we easily check that $$\begin{array}{lll}
b(1-ab)^m&=&b\big[\sum\limits_{i=0}^{m}\left(
\begin{array}{c}
m\\
i
\end{array}
\right)(-1)^i(ab)^i\big]\\
&=&\big[\sum\limits_{i=0}^{m}\left(
\begin{array}{c}
m\\
i
\end{array}
\right)(-1)^i(ba)^i]b\\
&=&(1-ba)^mb.
\end{array}$$ Then
we have $$(1-ab)^k-1=-ab[1+(1-ab)+\cdots +(1-ab)^{k-1}].$$ Hence,
$$(1-ab)^k=1-a[1+(1-ba)+\cdots +(1-ba)^{k-1}]b.$$
Likewise, $$\begin{array}{lll}
(1-ba)^k&=&1-ba[1+(1-ba)+\cdots +(1-ba)^{k-1}].
\end{array}$$
Since $(1+ab)^k$ has generalized Zhou inverse, then so has $1-a[1+(1-ba)+\cdots +(1-ba)^{k-1}]b$.
In light of Theorem 3.3, $1-ba[1+(1-ba)+\cdots +(1-ba)^{k-1}]$ has generalized Zhou inverse. This
implies that $(1+ba)^k$ has generalized Zhou inverse, as asserted.\end{proof}

\section{Zhou inverses}

 An element $a$ in $R$ is said to have Zhou inverse if there exists $b\in R$ such that $$b=bab, b\in comm^2(a), a^n-ab\in N(R)$$ for some $n\in \mathbb{N}$.  The preceding $b$ is unique, if such an element exists. An element $a$ in a ring $R$ has Drazin inverse if there exists $x\in comm(a)$ such that $x=xax, a^n=a^{n+1}x$ for some $n\in \mathbb{N}$. As is well known, $a\in R$ has Drazin inverse if and only if
 there exists $x\in coom^2(a)$ such that $x=xax, a-a^2x\in N(R)$. We now derive

\begin{thm} Let $R$ be a ring, and let $a\in R$. Then the following are equivalent:\end{thm}
\begin{enumerate}
\item [(1)]{\it $a$ has Zhou inverse.}
\vspace{-.5mm}
\item [(2)]{\it $a-a^{n+1}\in N(R)$ for some $n\in \mathbb{N}$.}
\end{enumerate}
\begin{proof} $(1)\Rightarrow (2)$ By hypothesis, there exists $b\in R$ such that $$b=bab, b\in comm^2(a), w:=a^n-ab\in N(R)$$ for some $n\in \mathbb{N}$.
Then $a^n=ab+w$, and so $a^{2n}=ab+(2ab+w)w$. Hence, $a^n-a^{2n}=(1-2ab-w)w\in N(R)$, and so $a^{n-1}(a-a^{n+1})=a^n-a^{2n}\in N(R)$.
Therefore $$\begin{array}{lll}
(a-a^{n+1})^n&=&[a^{n-1}(1-a^n)^{n-1}](a-a^{n+1})\\
&=&[a^{n-1}(a-a^{n+1})](1-a^n)^{n-1}\\
&\in& N(R).
\end{array}$$ Accordingly, $a^n-a^{n+1}\in N(R)$, as desired.

$(2)\Rightarrow (1)$ Write $(a-a^{n+1})^m=0$ for some $m\in \mathbb{N}$. Then $a^m=a^{m+1}f(a)$ for some $f(t)\in \mathbb{Z}[t]$.
In light of [Theorem]{C}, $R$ is periodic. That is, $a^k=a^{k+1}a^{l-1}$ for some $k,l\in \mathbb{N}$. So $a$ has Drazin inverse, and
then we have $b\in comm^2(a)$ such that $$b=bab, a-a^2b\in N(R).$$ Hence
$$(a^n-ab)(1-ab)=a^n(1-ab)=a^{n-1}(a-a^2b)\in N(R).$$ On the other hand,
$$(a^n-ab)ab=a^{n+1}b-ab=-(a-a^{n+1})b\in N(R).$$ Therefore $$a^n-ab=(a^n-ab)(1-ab)+(a^n-ab)ab\in N(R).$$
Accordingly, $a$ has Zhou inverse, as asserted.\end{proof}

\begin{cor} Let $R$ be a ring, and let $a,b,c,d\in R$ satisfy $bdb=bac, dbd=acd$. If $ac$ has Zhou inverse, then $bd$ has Zhou inverse.\end{cor}
\begin{proof} In view of Theorem 4.1, there exists some $n\in \mathbb{N}$ such that $ac-(ac)^{n+1}\in N(R)$.
As in the proof of Theorem 3.1, $(db)(ac)=(ac)(db)$. Hence, $db[ac-(ac)^{n+1}]\in N(R)$.
Thus we compute that $$\begin{array}{lll}
bd[bd-(bd)^{n+1}]&=&b(dbd)-bd(bd)^{n+1}\\
&=&b(ac)d-b(ac)^{n+1}d\\
&=&b[ac-(ac)^{n+1}]d\\
&\in& N(R).
\end{array}$$ Consequently, we have
$$[bd-(bd)^{n+1}]^2=bd[bd-(bd)^{n+1}][1-(bd)^n]\in N(R).$$
We infer that $bd-(bd)^{n+1}\in N(R)$. This completes the proof by Theorem 4.1.
\end{proof}

\begin{cor} Let $R$ be a ring, and let $a,b,c,d\in R$ satisfy $bdb=bac, dbd=acd$. If $ac$ has Zhou inverse, then $bd$ has Zhou inverse.
\end{cor}
\begin{proof} By virtue of Theorem 4.1, there exists some $n\in \mathbb{N}$ such that $1-ac-(1-ac)^{n+1}\in N(R)$.
Clearly, $$1-bd-(1-bd)^{n+1}=-bd[1+\sum\limits_{i=1}^{n+1}
(-1)^i\left(
\begin{array}{c}
n+1\\
i
\end{array}
\right)(bd)^{i-1}].$$ Likewise, $$1-ac-(1-ac)^{n+1}=-ac[1+\sum\limits_{i=1}^{n+1}
(-1)^i\left(
\begin{array}{c}
n+1\\
i
\end{array}
\right)(ac)^{i-1}].$$ Hence, we have
$$\begin{array}{ll}
&bd[1-bd-(1-bd)^{n+1}]\\
=&-b(dbd)[1+\sum\limits_{i=1}^{n+1}
(-1)^i\left(
\begin{array}{c}
n+1\\
i
\end{array}
\right)(bd)^{i-1}]\\
=&-b(acd)[1+\sum\limits_{i=1}^{n+1}
(-1)^i\left(
\begin{array}{c}
n+1\\
i
\end{array}
\right)(bd)^{i-1}]\\
=&-bac[1+\sum\limits_{i=1}^{n+1}
(-1)^i\left(
\begin{array}{c}
n+1\\
i
\end{array}
\right)(ac)^{i-1}]d.
\end{array}$$
Therefore we get $$\begin{array}{ll}
&[1-bd-(1-bd)^{n+1}]^2\\
=&bd[1-bd-(1-bd)^{n+1}][1+\sum\limits_{i=1}^{n+1}
(-1)^i\left(
\begin{array}{c}
n+1\\
i
\end{array}
\right)(bd)^{i-1}]\\
=&-bac[1+\sum\limits_{i=1}^{n+1}
(-1)^i\left(
\begin{array}{c}
n+1\\
i
\end{array}
\right)(ac)^{i-1}]\\
&d[1+\sum\limits_{i=1}^{n+1}
(-1)^i\left(
\begin{array}{c}
n+1\\
i
\end{array}
\right)(bd)^{i-1}]\\
=&-bac[1+\sum\limits_{i=1}^{n+1}
(-1)^i\left(
\begin{array}{c}
n+1\\
i
\end{array}
\right)(ac)^{i-1}]\\
&[1+\sum\limits_{i=1}^{n+1}
(-1)^i\left(
\begin{array}{c}
n+1\\
i
\end{array}
\right)(ac)^{i-1}]d\\
=&b[1-ac-(1-ac)^{n+1}][1+\sum\limits_{i=1}^{n+1}
(-1)^i\left(
\begin{array}{c}
n+1\\
i
\end{array}
\right)(ac)^{i-1}]d.
\end{array}$$
Since $(db)(ac)=(ac)(db$, we see that $$db[1-ac-(1-ac)^{n+1}][1+\sum\limits_{i=1}^{n+1}
(-1)^i\left(
\begin{array}{c}
n+1\\
i
\end{array}
\right)(ac)^{i-1}]\in N(R)$$ we have $$b[1-ac-(1-ac)^{n+1}][1+\sum\limits_{i=1}^{n+1}
(-1)^i\left(
\begin{array}{c}
n+1\\
i
\end{array}
\right)(ac)^{i-1}]d\in N(R).$$
Accordingly, $$[1-bd-(1-bd)^{n+1}]^2=bd[bd-(bd)^{n+1}][1-(bd)^n]\in N(R).$$
Therefore $1-bd-(1-bd)^{n+1}\in N(R)$, as required.\end{proof}

Let $R$ be a ring, and let $a,b\in R$. Then $ab$ has Zhou inverse if and only if $ba$ has Zhou inverse, $1-ab$ has Zhou inverse if and only if $1-ba$ has Zhou inverse.
These are immediately followed by Corollary 4.2 and Corollary 4.3. Let $\mathcal{A}$ be a Banach algebra, and let $a\in \mathcal{A}$. By using Theorem 4.1 and Corollary 2.10, $a$ has generalized Zhou inverse if and only if $\overline{a}\in \mathcal{A}/J(\mathcal{A})$ has Zhou inverse.

\begin{thm} Let $R$ be a ring, and let $a\in R$. Then the following are equivalent:\end{thm}
\begin{enumerate}
\item [(1)]{\it $a$ has Zhou inverse.}
\vspace{-.5mm}
\item [(2)]{\it There exists $p^2=p\in comm^2(a)$ such that $a^n-p\in N(R)$ for some $n\in \mathbb{N}$.}
\vspace{-.5mm}
\item [(3)]{\it There exists $p^2=p\in comm(a)$ such that $a^n-p\in N(R)$.}
\vspace{-.5mm}
\end{enumerate}
\begin{proof} $(1)\Rightarrow (2)$ By hypothesis, there exists $b\in comm^2(a)$ such that $b=bab, a^n-ab\in N(R)$ for some $n\in \mathbb{N}$. Set $p=ab$. Then
$p^2=p\in comm^2(a)$ and $a^n-p\in N(R)$, as desired.

$(2)\Rightarrow (3)$ This is trivial.

$(3)\Rightarrow (1)$ Set $w=a^n-p$. Then $a^n=p+w$ and $pw=wp$, and so $a^{2n}=p+(2p+w)w$. We infer that
$a^n-a^{2n}=(1-2p-w)w\in N(R)$. Hence $a^{n-1}(a-a^{n+1})\in N(R)$. This shows that $(a-a^{n+1})^n=a^{n-1}(a-a^{n+1})(1-a^{n+1})^{n-1}\in N(R)$; whence, $a-a^{n+1}\in N(R)$. Therefore $a$ has Zhou inverse by Theorem 4.1.\end{proof}

\begin{cor} Let $R$ be a ring, and let $a\in R$. Then the following are equivalent:\end{cor}
\begin{enumerate}
\item [(1)]{\it $a$ has Zhou inverse.}
\vspace{-.5mm}
\item [(2)]{\it There exists a unique idempotent $p\in comm^2(a)$ such that $a^n-p\in N(R)$ for some $n\in \mathbb{N}$}
\end{enumerate}
\begin{proof} $(1)\Rightarrow (2)$ In view of Theorem 4.4, there exists $p^2=p\in comm^2(a)$ such that
$a^n-p\in N(R)$ for some $n\in \mathbb{N}$. Suppose that there exists $q^2=q\in comm^2(a)$ such that
$a^m-q\in N(R)$. Then $a^{mn}-p, a^{mn}-q\in N(R)$ with $pq=qp$. Set $k=mn$. We check that
$$\begin{array}{lll}
1-p(1-q)&=&1-p(a^k+1-p)^{-1}(a^k+1-p)(1-q)\\
&=&1-p(a^k+1-p)^{-1}a^k(1-q)\\
&=&1-ca^k(1-q),
\end{array}$$ where $c=p(a^k+1-p)^{-1}$. Since $p\in comm^2(a)$, we
have $c\in comm(a^kq)$. It follows from $a^k(1-q)\in N(R)$ that
$1-ca^k(1-q)\in U(R)$. Since $1-p(1-q)$ is an idempotent, we see that
$1-p(1-q)=1$; hence, $p=qp$. Likewise, $q=pq$.
Accordingly, $p=qp=pq=q$, as required.

$(2)\Rightarrow (1)$ This is obvious by Theorem 4.4.\end{proof}


\vskip10mm\begin{thebibliography}{99} \bibitem{C} M. Chacron, On a theorem of Herstein, {\it Canad. J. Math.}, {\bf 21}(1969), 1348-1353.

\bibitem{CM1} H. Chen and M. Sheibani, Generalized Hirano inverses in Banach algebras, {\it Filomat}, {\bf 33}(2019), 6239-6249.

\bibitem{CM} H. Chen and M. Sheibani, Generalized Cline's formula and commonl spectral property  {\it J. Algebra Appl.}, DOI: 10.1142/S0219498821500948 (to appear).

\bibitem{CC} J. Cui and J. Chen, Pseudopolar matrix rings over local rings, {\it J. Algebra Appl.}, {\bf 13}(2014), DOI: 10.1142/S0219498813501090.

\bibitem{K} M.T. Kosan, T. Yildirim and Y. Zhou, Rings with $x^n-x$ nilpotent, {\it J. Algebra Appl.}, {\bf 19}(2020), DOI: 10.1142/S0219498820500656.

\bibitem{LCC} Y. Liao; J. Chen and J. Cui, Cline's formula for the generalized Drazin inverse, {\it Bull. Malays. Math. Sci. Soc.}, {\bf 37}(2014), 37-42.

\bibitem{M} D. Mosic, Extensions of Jacobson's lemma for Drazin inverses, {\it Aequat. Math.}, {\bf 91}(2017), 419-428.

\bibitem{M1} D. Mosic, The generalized and pseudo n-strong Drazin inverses in rings, {\it Linear Multilinear Algebra}, April 2019, DOI: 10.1080/03081087.2019.1599806.

\bibitem{R} C. Rickart, General Theory of Banach Algebra, Van Nostrand, N.Y. 1960.

\bibitem{W} Z. Wang and J. Chen, Pseudo Drazin inverses in associative rings and Banach algebras, {\it Linear Algebra Appl.}, {\bf 437}(2012), 1332-1345.

\bibitem{Y} Z.L. Ying; T. Kosan and Y. Zhou, Rings in which every element is a sum of two tripotents, {\it Canad. Math. Bull.}, {\bf 59}(2016), 1-15.

\bibitem{Z} Y. Zhou, Rings in which elements are sums of nilpotents, idempotents and tripotents, {\it J. Algebra Appl.}, {\bf 17}(2018), DOI: 10.1142/S0219498818500093.

\bibitem{Z2} H. Zhu; J. Chen and P. Patricio, Representations for the pseudo Drazin inverse of elements in a Banach algebra, {\it Taiwanese J. Math.},
{\bf 19}(2015), 349-362.

\bibitem{ZC} G.F. Zhuang; J.L. Chen and J. Cui, Jacobson's Lemma for the generalized Drazin inverse, {\it Linear Algebra Appl.}, {\bf 436}(2012), 742-746.

\bibitem{ZC1} H. Zou and J. Chen, On the pseudo Drazin inverse of the sum of two elements in a Banach algebra, {\it Filomat}, {\bf 31}(2017), 2011-2022.

\bibitem{ZC2} H. Zou; J. Chen and H. Zhu, Characterizations for the n-strong Drazin invertibility in a ring, {\it J. Algebra Appl.}, July, 2020.

\bibitem{ZC3} H. Zou; D. Mosic; K. Zuo and Y. Chen, On the n-strong Drazin invertibility in rings, {\it Turk. J. Math.}, {bf 43}(2019), 2659-2679.

\end{thebibliography}
\end{document}